\documentclass[aap]{amsart}
\usepackage{amsfonts,amsmath,amssymb,amscd}

\usepackage{color}   

\begin{document}

\title[Scale-free laws via thinning]{Scale-free and power law distributions via 
fixed points and convergence
of  (thinning and conditioning) transformations
}
 
\author[Arratia]{Richard Arratia}
\address[Richard Arratia]{Department of Mathematics, University of Southern California,
Los Angeles CA 90089.}
\email{rarratia@usc.edu}
\urladdr{http://dornsife.usc.edu/cf/faculty-and-staff/faculty.cfm?pid=1003062}

\author[Liggett]{Thomas M. Liggett}
\address[Thomas M. Liggett]{Department of Mathematics, University of California
Los Angeles CA 90095-1555.}
\email{tml@math.ucla.edu}
\urladdr{http://www.math.ucla.edu/~tml/}
 
\author[Williamson]{Malcolm J. WIlliamson}
\address[Malcolm J. Williamson]{Center for Communications Research, 4320 Westerra Ct., 
San Diego, CA 92121-1969.}
\email{malcolm@ccrwest.org}
 
 
   
\date{May 27, 2014}

\newtheorem{theorem}{Theorem}
\newtheorem{lemma}[theorem]{Lemma}

\newtheorem{prop}[theorem]{Proposition}
\newtheorem{cor}[theorem]{Corollary}
\theoremstyle{definition}
\newtheorem{remark}{Remark}

\newcommand{\ignore}[1]{}
\def\e{\mathbb{E\,}}
\def\p{\mathbb{P}}
\def\er{Erd\H{o}s-R\'enyi }
\def\Dfixed{D^{\rm fixed}_\beta}
\def\M{G}  
\def\x{s}   

\begin{abstract} 
In discrete contexts such as the degree distribution for a graph,
\emph{scale-free} has traditionally been \emph{defined} to be \emph{power-law}.
We propose a reasonable interpretation of \emph{scale-free}, namely, invariance under the  transformation of $p$-thinning, followed by conditioning on being positive.

For each $\beta \in (1,2)$, we show that there is a unique distribution which is a fixed point of this transformation;
the distribution is power-law-$\beta$, and different from the usual Yule--Simon power law-$\beta$ that arises in
preferential attachment models.

In addition to characterizing these fixed points, we prove convergence results for iterates of the transformation.
\end{abstract}
 
\maketitle


\section{Introduction and statement of results}

In the context of of random graphs, many authors define the term {\it scale-free}
to mean that the degree distribution follows a power law -- see for example \cite{ab,math}.
In this paper, we adopt a different point of view, in which scale-free means that the
degree distribution is invariant under a natural transformation on the graph. As we will see, the power law property is
then a consequence of this definition.

To motivate our transformation, consider a continuous random variable $X\geq 1$.
It appears natural to say that its distribution is scale-free if $cX$ conditioned on $cX\geq 1$
has the same same distribution as $X$, i.e.,
$$\p(X\geq x)=\p(cX\geq x\mid cX\geq 1).$$
It is not hard to check that the only such distributions are the Pareto
distributions
$$\p(X\geq x)=x^{-\alpha},\quad x\geq 1.$$
See \cite{krishnaji,patil} for similar observations. One can also consider convergence to these fixed points, and
easily show that
$$\lim_{c\rightarrow 0}\p(cX\geq x\mid cX\geq 1)=x^{-\alpha},\quad x\geq 1$$
if and only if the tail probabilities $\p(X\geq x)$ are of the form $L(x)x^{-\alpha}$, where $L$ is slowly varying.

We consider now a discrete analogue of this setup. If $D$ is a nonnegative integer valued random variable, 
$cD$ is no longer integer valued, so we replace multiplication by thinning. A $p-$thinning of $D$ is
defined by
\begin{equation}\label{pthin}S_D=\sum_{i=1}^DX_i,\end{equation}
where $X_i$ are i.i.d. Bernoulli ($p$) random variables that are independent of $D$. In terms of the probability generating function $G_D(s)=\e s^D$, this becomes
\begin{equation}\label{M thinned} G_{S_D}(s)=G_D(1-p+ps)=G_D(1-p(1-s)).\end{equation}
In the graph context, this corresponds to thinning by edges.

We are concerned here
with fixed points of the transformation $T=T_p=T_{p,m}$ given by
$$T:D\rightarrow (S_D\mid S_D\geq m)$$
where $m$ is an integer $\geq 1$, and convergence to these fixed points. (The case $m=1$ is the most natural.)

There are other contexts in which fixed points and convergence of transformations that are the composition of two operations that change a distribution in opposite directions have been studied. Examples are 
\cite{aldous,fixed}.

Similar questions
for other families of transformations acting on discrete distributions have been studied before
-- see \cite{Bouzar,Davydov,Vervaat} for example. The main feature that distinguishes our setting from these others
is the conditioning. 

 We will use two forms
of the power-law-$\beta$ property:
\begin{align}\p(D=k)&\sim ck^{-\beta},\label{point}\\\p(D\geq n)&\sim L(n)n^{1-\beta}
\label{tail slow}\end{align}
where $\beta>1$ and $L$ is slowly varying. The latter property is known as regular variation.
Our characterization of fixed points is the following. It is proved in Section \ref{sect unique}.

\begin{theorem}\label{thm unique}
Let $m$ be a positive integer, and let $D$
be a nonnegative integer valued random variable,
with $\p(D \ge m)>0$.
The following are equivalent:

\begin{itemize}

\item  The distribution of $D$ 
is fixed by the transformation $D \mapsto T_{p,m} D$  for all $p \in (0,1)$.

\item Either
$D \equiv m$ is constant, or else $D$ has power-law-$\beta$ distribution \eqref{point},
with $\beta=\alpha+1$, $0 <\alpha <m$, 
$\p(D<m)=0$, and
\begin{equation}\label{unique display}
 \p(D=k+1)/\p(D=k)=(k-\alpha)/(k+1)  \text{ for } k \ge m.
\end{equation}
\end{itemize} 

\end{theorem}

For the convergence results, we consider separately the cases of nontrivial and trivial fixed points.
For the motivation for taking $p\downarrow 0$ in these results, see Remark \ref{sect iterate} in the next section.

\begin{theorem}\label{main theorem} Suppose 
the distribution of $D$ is power-law-$\beta$, as specified by
 \eqref{tail slow}.
Then for every integer
$k \ge \beta$
\begin{equation}\label{was 3}
\lim_{p\rightarrow 0+}\frac{\p(S_D= k)}{\p(S_D=k-1)}=\frac{k-\beta}k.
\end{equation}
\end{theorem}

\begin{theorem}\label{thm converge 2}
Take $m\geq\beta-1$, and suppose the distribution of 
$D$ is such that \eqref{was 3} holds
for $k\geq \beta$. Then the distributions of
$(S_D\mid S_D\geq m)$ are tight as $p\downarrow 0$. It follows that these
distributions have a limit as $p\downarrow 0$, which is the fixed point
described in \eqref{unique display} in case $\beta< m+1$, or $\p(D=m)=1$ in case $\beta=m+1$.
\end{theorem}

\begin{theorem}\label{theorem trivial}Suppose $ED^{k-1}<\infty$. Then
\begin{equation}
\lim_{p\rightarrow 0}\frac{\p(S_D\geq k)}{\p(S_D=k-1)}=0,\end{equation}
provided that the denominator above is strictly positive.
As a consequence, if $ED^m<\infty$ and $\p(D\geq m)>0$, then
$$\lim_{p\rightarrow 0}\p(S_D=m\mid S_D\geq m)=1.$$
\end{theorem}

These three results are proved in Sections \ref{sect converge} and \ref{sect trivial}.
In the final section, we prove that the nontrivial fixed points are infinitely divisible. 

\section{The tranformations $T_{p,m}$ and their fixed points}

If $D$ is a nonnegative integer valued random variable and $0<p<1$,
the $p$-thinning $S_D$ of $D$, defined by \eqref{pthin}, has,
using  the notation $(z)_k = z(z-1)\cdots (z- k + 1)$ for the falling
product,
\begin{equation}\label{was 1}\begin{aligned}
 \p(S_D=n)&=
\sum_{l=n}^{\infty} \p(D=l) \binom lnp^n(1-p)^{l-n}\\&
=\bigg(\frac p{1-p}\bigg)^n\frac 1{n!}\sum_{l=n}^{\infty}
(l)_n(1-p)^l\p(D=l).\end{aligned}
\end{equation}
\ignore{
is the random variable given by
$$S_D=X_1+\cdots +X_D,$$
where $X_1, X_2,...$ are i.i.d. random variables (also
independent of $D$) with $\p(X_i=1)=p,
\p(X_i=0)=1-p$. The distribution of $S_D$ is then given by
\begin{multline}\label{was 1}
 \p(S_D=n)=
\sum_{l=n}^{\infty} \p(D=l) \binom lnp^n(1-p)^{l-n}\\
=\bigg(\frac p{1-p}\bigg)^n\frac 1{n!}\sum_{l=n}^{\infty}
(l)_n(1-p)^l\p(D=l).\\
\end{multline}
}

Fix an integer $m=1,2,\ldots$.  For $p \in (0,1)$, the transformations $T
\equiv T_p \equiv T_{p,m}$ 
for which we consider fixed points and convergence
of iterates are given by
\begin{equation}\label{def T}
\p(TD=l)=\p(S_D=l\mid S_D\geq m).
\end{equation}

In Section \ref{sect unique}, we will prove that the fixed points of the transformation
are precisely those 
described 
by \eqref{G at m=1} -- \eqref{our ratio}
below, and in Section \ref{sect converge} and \ref{sect trivial}
we will prove results where these fixed points arise as limits of iterates of the transformation.

\begin{remark}
We are referring here to distributions that are fixed points for all
$p$, not just for some $p$. It would be interesting to know whether
these are the only fixed points for a given $p$.
\end{remark}

For $m=1,2,\ldots$, the distribution with  $\p(D=m)=1$ is a trivial fixed point of $T_{p,m}$. 
For $m=1$, all nontrivial fixed points have the form:  for some $\alpha \in (0,1)$, 
\begin{equation}\label{G at m=1}
  G_D(s) := \e s^D = 1-(1-s)^{\alpha} =: \sum_{k \ge 0} c_k(\alpha) s^k.
\end{equation}
The right hand side of \eqref{G at m=1} \emph{defines} $c_k(\alpha)$ to be the coefficient of $s^k$ in $1-(1-s)^{\alpha}$, 
so 
\ignore{
$$
c_0(\alpha) = 0, c_1(\alpha) =\alpha, c_2(\alpha) =\frac{\alpha(1-\alpha)}{2},c_3(\alpha) =\frac{\alpha(1-\alpha)(2-\alpha)}{6};
$$
for $k \ge 1$, $c_k(\alpha) = (-1)^{k-1}(\alpha)_k/k!$, 
\begin{equation}\label{density at m=1}
  c_k(\alpha) = \frac{(-1)^{k-1}(\alpha)_k}{k!};
\end{equation}
and then specifically for $m=1$, with the restriction $\alpha \in (0,1)$,
$$
   \p(D=k)=c_k(\alpha), \ k=1,2,\ldots \ .
$$
}
for $k \ge 1$, $c_k(\alpha) = (-1)^{k-1}(\alpha)_k/k!$, and for $m=1$, with the restriction $\alpha \in (0,1)$,
$   \p(D=k)=c_k(\alpha), \ k=1,2,\ldots \ .$

In general, for $m=1,2,\ldots$ and $\alpha \in (0,m)$ there is a nontrivial fixed point for $T_{p,m}$, 
which is power-law-$\beta$ for $\beta = 1+\alpha$, with
\begin{equation}\label{G at m}
  G_D(s) := \e s^D = \frac{1-(1-s)^{\alpha} - \sum_{1 \le k < m} c_k(\alpha) s^k }{1-\sum_{1 \le k < m} c_k(\alpha)},
\end{equation}  
and this gives all nontrivial fixed points 
of $T_{p,m}$.
\ignore{ 
Thus, the special case $m=1$ of \eqref{G at m} was given by \eqref{G at m=1}, under the restriction $0 < \alpha < 1$; the special case $m=2$ of \eqref{G at m=1} is
$$
G_D(s) := \e s^D = \frac{1-(1-s)^{\alpha} - \alpha s}{1-\alpha},  \text{ for } 0 < \alpha <2;
$$
and the special case $m=3$ of \eqref{G at m=1} is
$$
G_D(s) := \e s^D = \frac{1-(1-s)^{\alpha} - \alpha s - \frac{\alpha(1-\alpha)}2 s^2}{1-\alpha- \frac{\alpha(1-\alpha)}2} ,   \text{ for } 0 < \alpha <3.
$$
}
A unified description of the fixed points (for all $p$)  of $T_{p,m}$, including both the trivial fixed point, obtained by taking $\alpha = m$, is:  $1+\alpha =\beta \in (1,m+1]$, $\p(D \in \{m,m+1,m+2,\ldots \})=1$, $\p(D=m)>0$, and
\begin{equation}\label{our ratio alpha}
\frac{\p(D=k+1)}{\p(D=k)}=\frac{k-\alpha}{k+1},\quad k \ge m.
\end{equation}
or equivalently, shifting the dummy variable $k$ by 1,
\begin{equation}\label{our ratio}
\frac{\p(D=k)}{\p(D=k-1)}=\frac{k-\beta}{k},\quad k>m.
\end{equation}

The Yule--Simon distribution for power-law-$\beta$ has point
probabilities given by $\p(D=k) = (\beta-1) \, \Gamma(k)
\Gamma(\beta)/\Gamma(k+\beta)$, and hence ratios 
\begin{equation}\label{yule ratio}
   \frac{\p(D=k)}{\p(D=k-1)}= \frac{k-1}{k-1+\beta}. 
\end{equation}
    In comparison with \eqref{our ratio}, both
formulas have denominator minus numerator $= \beta$, for every $k$, but 
for non-integer $\beta$, \eqref{our ratio}
has the integer in the denominator, while the Yule--Simon ratio \eqref{yule ratio} has the integer
in the numerator.

\begin{remark}\label{sect
    iterate}
For each $m=1,2,\ldots$, it is true that for all $p,q \in (0,1)$ one
has $T_q \circ T_p = T_{pq}$; we omit the easy proof.
It then follows that  the $k$-fold iterate $(T_q)^k$ 
of $T_q $ is $T_p$ with $p=q^k$.  
\ignore{
It may or may not be intuitively obvious that the four steps:
$p$-thinning to get $S_D$, then conditioning on $S_D \ge m$ to get $T_p D$,
then $q$-thinning --- to get say $Y$, and conditioning again, on $Y
\ge m$, to get $T_q T_p D$,
have the same effect as the
two steps:  $pq$-thin, then condition, to get $T_{pq} D$.   
The intuition is reasonable, 
since
the event $Y \ge m$ is a subset of the event $S_D \ge m$.  
For the case $m=1$, an easy way to see that $T_p $ followed by $T_q $ equals $T_{pq} $ 
is to use the  ``doubly complemented''
generating
functions from the end of Section \ref{sect p coin}, 
for which the
distribution of  $T_p D \equiv T_{p,1} D$ is determined by
\begin{equation}\label{iterate m=1}
   N_{T_p D}(s) = N_D(ps) \, / \, N_D(p).
\end{equation}
This allows a proof via the calculation
$$
N_{T_q T_p D}(s) =  \frac{N_{T_p D}(qs)}{ N_{T_p D}(q)} =
 \frac{ N_D(pqs) \, / \, N_D(p)}{ N_D(pq) \, / \, N_D(p)} = 
\frac{ N_D(pqs) }{ N_D(pq)} =  N_{T_{pq} D}(s).
$$
That $T_q \circ T_p = T_{pq}$  for all $m$ is true, and can be proved
using a well-known coupling $p$-coins, $q$-coins, and $pq$-coins; we
omit the details of this coupling proof.

From $T_s \circ T_t = T_{st}$ it follows that the $k$-fold iterate $(T_q)^k$ 
of $T_q $ is $T_p$ with $p=q^k$.  
}
Theorem 
\ref{thm converge 2} 
allows $p \to 0$ with only the restriction $p>0$, 
and the special
case
where $p$ goes to zero along a geometric sequence $q^k$ yields
convergence for iterates of the transformation $T_q$, for one fixed $q$.
\end{remark}

\section{Uniqueness}\label{sect unique}

The goal is to show that, for $m=1,2,\ldots$,
any distribution $D$ on the nonnengative integers
 which is unchanged by $p$-thinning
followed by conditioning on being at least $m$, for all
$p \in (0,1)$,
is either
the constant $D \equiv m$ or else, as specified by \eqref{our ratio}, the law with $1 <\beta <m+1$
and ratios
$\p(D=k)/\p(D=k-1)=(k-\beta)/k$
for $k \ge m+1$.   

\begin{lemma}\label{trivial lemma}
 Suppose  $A$ and $B$ are two nonnegative integer valued random
variables with probability generating functions $\M_A$, $\M_B$.
Let $m$ be a positive integer.   Assume $\p(A \ge m) > 0$ and
$\p(B \ge m) > 0$. Consider the statements

\begin{description}

\item[(a)] $\p(A=k)=\p(B=k)$ for all $k \ge m$.

\item[(b)]  $(A|A\ge m)$ and $(B|B \ge m)$ have the same distribution.  

\item[(c)] $\M_A^{(m)}(s)=\M_B^{(m)}(s)$
for all $s \in [0,1)$.

\item[(d)] $\M_A^{(m)}(s)=c \, \M_B^{(m)}(s)$
for all $s \in [0,1)$, for some constant $c>0$.

\end{description}
(Here $G_A^{(m)}(s)$ denotes the $m$th derivative of $G_A(s)$.) Then (a) if and only if (c),   and (b) if and only if (d).

\end{lemma}

\begin{proof}

Let  $a_k := \p(A=k)$ and $b_k=\p(B=k)$ so that 
$\M_A(s)=\sum_{k \ge 0} a_k s^k$ and likewise for $\M_B$.
These are power series with radius of convergence $\ge 1$, hence
differentiable term-by-term, with
$\M_A^{(m)}(s)=\sum_{k \ge m} k_{(m)} a_k s^{k-m}$ for $|s|<1$,
and likewise for $\M_B$.
This immediately shows that (a) implies (c);  
to see that (c) implies (a), given $k \ge m$, differentiate $k-m$ times and evaluate at
$s=0$.

The equivalence of (b) and (d) follows, 
with $c=\p(B \ge m)/\p(A \ge m)$.

\end{proof}

We apply this with $A=D$ and $B=S_D$. We are looking for a fixed point
of $D \mapsto T D$, where $T D\equiv T_{p,m} D:=(S_D|S_D \ge m)$ and
$S_D$ is the $p$-thinning of $D$.
Since $1=\p(TD \ge m)$, we can have $D$ and $T D$ equal in distribution
only
if $1=\p(D \ge m)$.   Thus we assume that $1=\p(D \ge m)$, so that
$D = (D|D \ge m)$, and now we have a fixed point of
$D \mapsto T D$ if and only if $(D|D \ge m) = (S_D|S_D \ge m)$.  
Combine Lemma \ref{trivial lemma} with  \eqref{M thinned}, 
so that the two generating functions of interest are
$\M_A(s) = \M(s)$ and $\M_B(s) =\M(1-p(1-s))$.

 Write $f$ for the $m$th derivative
of $\M$, so that $\M_B^{(m)}(s)=(\M(1-p+ps))^{(m)}  
=p^m \, f(1-p(1-s))$.  Assuming that $1=\p(D \ge m)$,
we have a fixed point of $D \mapsto T_{p,m} D$ if and only if
$$
f(s)=c \, p^m \times f(1-p(1-s)), \ \ \mbox{for all } s \in [0,1).
$$

\begin{lemma}\label{linear lemma}
Let $f$ be a continuous function from $[0,1)$ to
$(0,\infty)$, with $f(0)=1$,
and let $p \mapsto c(p)$ be any function on $(0,1)$.  If
\begin{equation}\label{f}
\forall p \in (0,1), 
\forall s \in [0,1), \ \ \  \ \ f(1-p(1-s)) = c(p) f(s),
\end{equation}
then for some constant $d$ we have $f(s)=(1-s)^{-d}$.
\end{lemma}

\begin{proof}
First let $s=1-t$ so that \eqref{f} becomes
$$
\forall p \in (0,1), 
\forall t \in (0,1], \ \ \  \ \ f(1-pt) = c(p) f(1-t),
$$
and then consider $g(t) := f(1-t)$ so that the system to solve
becomes
\begin{equation}\label{g}
\forall p \in (0,1), 
\forall t \in (0,1], \ \ \  \ \ g(pt) = c(p) g(t),
\end{equation}
with $g(1)=1$.
Plugging in $t=1$ we see that $c(p)=g(p)$, and \eqref{g}
becomes $g(pt)=g(p)g(t)$.  It follows that $g(u)=u^{-d}$ for some $d$.
\end{proof}

\begin{proof}[Proof of Theorem \ref{thm unique}]
Start by assuming that $D$ is a fixed point.  
We combine Lemmas \ref{trivial lemma} and \ref{linear lemma}, as in
the remarks before Lemma \ref{linear lemma}, so that $\M(s)=\e s^D$, $f$ is
the $m$th derivative of $\M$, and
the conclusion of Lemma \ref{linear lemma} applied to $f(s)/f(0)$ is that 
$f(s)=c \, (1-s)^{-d}$ with $c>0$. [We have $c=f(0)>0$ because 
$\p(D \ge m)>0$ implies $\p(S_D = m)>0$, hence $\p(D=m|D \ge m) =
\p(S_D=m |S_D \ge m)>0$, hence $c=m! \p(D=m)>0$.]

In case $d=0$, we have $f$ is constant and $D \equiv m$.  We cannot
have $d$ negative, since then the coefficient of $s^1$ in $f$ is $d$,
while $\M$ has nonnegative coefficients.  In case $d>0$,
writing $[s^k]f(s)$ for the coefficient of $s^k$ in $f$,
so that $[s^k]\M(s)=\p(D=k)$, we have for $k \ge m$
$$
 k_{(m)} \p(D=k) = [s^{k-m}]f(s) =  [s^{k-m}] (c \, (1-s)^{-d}) =
c \, (-1)^{k-m} \, \frac{ (-d)_{(k-m)}}{(k-m)!}.
$$
Hence for $k \ge m$
$$
\p(D=k)  = c \, (-1)^{k-m} \, \frac{ (-d)_{(k-m)}}{k!}
$$
and 
$$
\frac{\p(D=k+1)}{\p(D=k)} = \frac{-(-d-(k-m))}{k+1} = \frac{k-\alpha}{k+1},
$$
with $\alpha = m-d <m$.
The requirement $\sum \p(D=k) < \infty$ implies that $\alpha > 0$.

The implication in the opposite direction is easy, again by
combining Lemmas \ref{trivial lemma} and \ref{linear lemma}.
\end{proof}

\section{Convergence to nontrivial fixed points}
\label{sect converge}

Before proving Theorem \ref{main theorem}, we state part of a Tauberian theorem that can
be found on page 447 of \cite{fellerII}. Many other Tauberian theorems
can be found in  \cite{regular}.
\begin{theorem}
Let $q_l\geq 0$ and suppose
$Q(s)=\sum_{l=0}^{\infty}q_l s^l$
converges for $0\leq s <1$. If $L$ is slowly varying, $\rho> 0$,
and 
$q_l\sim l^{\rho-1}L(l),$
then
$$Q(s)\sim\frac{\Gamma(\rho)}{(1-s)^{\rho}}L\bigg(\frac 1{1-s}\bigg)
\text{ as } s\uparrow 1. $$
\end{theorem}

\begin{proof} [Proof of Theorem \ref{main theorem}.]

Write $H(k)=\p(D\geq k)$, so that \eqref{tail slow} gives  $H(k)=k^{1-\beta}L(k)$, where $L$ is slowly varying. Sum by parts, make a change of variables in the second sum below, and apply the Tauberian theorem to each of the resulting sums. 
By (\ref{was 1}),

\ignore{
$$\begin{aligned}\sum_{l=k}^\infty l(l-1)&\cdots(l-k+1)(1-p)^l\p(D=l)=\sum_{l=k}^\infty l(l-1)\cdots(l-k+1)(1-p)^l[H(l)-H(l+1)]\\
&=\sum_{l=k}^\infty l(l-1)\cdots(l-k+1)(1-p)^lH(l)-\sum_{l=k+1}^\infty (l-1)\cdots(l-k)(1-p)^{l-1}H(l)\\
&=k!(1-p)^kH(k)+\sum_{l=k+1}^\infty (l-1)\cdots(l-k+1)(k-lp)(1-p)^{l-1}H(l)\\
&=k!(1-p)^kH(k)+k\sum_{l=k+1}^\infty (l-1)\cdots(l-k+1)(1-p)^{l-1}H(l)\\
&\hskip 1in -p\sum_{l=k+1}^\infty l(l-1)\cdots(l-k+1)(1-p)^{l-1}H(l)\\
&\sim k\Gamma(k-\beta+1)p^{\beta-k-1}L(p^{-1})-\Gamma(k-\beta+2)p^{\beta-k-1}L(p^{-1})\\
&=\Gamma(k-\beta+1)(\beta-1)p^{\beta-k-1}L(p^{-1}),\end{aligned}$$
{\bf  REPEAT, so that one can check typesetting}
}

$$\begin{aligned}P(S_D=k)k!\bigg(\frac{1-p}p\bigg)^k&=\sum_{l=k}^\infty (l)_k(1-p)^l\p(D=l)\\&=\sum_{l=k}^\infty (l)_k(1-p)^l[H(l)-H(l+1)]\\
&=\sum_{l=k}^\infty (l)_k(1-p)^lH(l)-\sum_{l=k+1}^\infty (l-1)_k(1-p)^{l-1}H(l)\\
&=k!(1-p)^kH(k)+\sum_{l=k+1}^\infty (l-1)_{k-1}(k-lp)(1-p)^{l-1}H(l)\\
&=k!(1-p)^kH(k)+k\sum_{l=k+1}^\infty (l-1)_{k-1}(1-p)^{l-1}H(l)\\
&\hskip 1in -p\sum_{l=k+1}^\infty (l)_k(1-p)^{l-1}H(l)\\
&\sim k\Gamma(k-\beta+1)p^{\beta-k-1}L(p^{-1})-\Gamma(k-\beta+2)p^{\beta-k-1}L(p^{-1})\\
&=\Gamma(k-\beta+1)(\beta-1)p^{\beta-k-1}L(p^{-1}),\end{aligned}$$
provided that $k-\beta+1>0$. This gives \eqref{was 3} if $k>\beta$. If $k=\beta$, the above computation with $k$ replaced by $k-1$ gives
$$\sum_{l=k-1}^\infty l(l-1)\cdots(l-k+2)(1-p)^l\p(D=l)\sim(k-1)!H(k-1)+(k-1)L^*(p^{-1}),$$
so \eqref{was 3} holds in this case as well.
\end{proof}

Convergence of
the ratios of probabilities in \eqref{was 3} does not immediately imply tightness of the
distributions of $(S_D\mid S_D\geq m)$ as $p\downarrow 0$. This tightness is
needed to conclude that the iterates of the transformation converge to the
appropriate fixed point.
We therefore now turn our attention to that issue.

\begin{proof}[Proof of Theorem \ref{thm converge 2}]
 Tightness of these conditional distributions means that
\begin{equation}\label{was 5}
\lim_{k\rightarrow\infty}\limsup_{p\rightarrow 0+}\frac {\p(S_D\geq k)}
{\p(S_D\geq m)}=0.\end{equation}
Thus we need to deduce the asymptotics of ratios of tail probabilities from
the asymptotics of ratios of point probabilities. 

A key identity that allows
for this transition is 
\begin{equation}\label{was 6}
\frac d{dp}\p(S_D\geq k)=kp^{-1}\p(S_D=k).\end{equation}
Students of the theory of percolation will recognize this as a very
simple form of Russo's formula -- see page 35 of \cite{grimmett}, 
for
example. 
 The proof of \eqref{was 6} 
 is also simple:
Use \eqref{was 1} to write
\begin{equation}\label{was 7}
\p(S_D\geq k)=\sum_{l=k}^{\infty}\p(D=l)\bigg[1-\sum_{n=0}^{k-1}\binom ln
p^n(1-p)^{l-n}\bigg].\end{equation} 
Differentiating gives
$$\frac d{dp}\p(S_D\geq k)=p^{-1}\sum_{l=k}^{\infty}\p(D=l)\sum_{n=0}^{k-1}
\binom ln p^n(1-p)^{l-n-1}(lp-n).$$
To prove \eqref{was 6} 
one needs to check
\begin{equation}\label{was 8}
\sum_{n=0}^{k-1}
\binom ln p^n(1-p)^{l-n-1}(lp-n)=k\binom lk p^k(1-p)^{l-k}.\end{equation}
The easiest way to check this is to note that the two sides of
\eqref{was 8} 
agree for $k=0$, and  differences of the two sides of \eqref{was 8} 
for successive values of $k$ also agree. 
\ignore{
(A probabilistic proof of
\eqref{was 6} 
can be given via a coupling argument: Let $U_1, U_2,...$ be i.i.d.
random variables that are uniformly distributed on $[0,1]$. Take the
Bernoulli random variables $X_i$ that are used to define the $p$-thinning
$S_D$ to be $X_i=1_{U_i\leq p}$. This constructs $S_D$ for different
$p$'s on a common probability space, and hence allows one to write the
difference quotient whose limit is the left side of \eqref{was 6} 
in terms of
the probability of an event on that probability space.)
}

By L'Hospital's Rule, whenever \eqref{was 3} holds, it follows from
\eqref{was 6} 
that
\begin{equation}\label{was 9}
\lim_{p\rightarrow 0+}\frac{\p(S_D\geq k)}{\p(S_D\geq
    k-1)}=\frac{k-\beta}{k-1}.
\end{equation}
Using \eqref{was 9} 
repeatedly gives
$$\lim_{p\rightarrow 0+}\frac{\p(S_D\geq m+k)}{\p(S_D\geq m)}=\prod_{j=1}^k\frac{m+j-\beta}
{m+j-1}.$$
Now 
\eqref{was 5} follows from this and the fact that
$\sum_j(\beta-1)(m+j-1)=\infty.$
\end{proof}

\section{Convergence to trivial fixed points}\label{sect trivial}

Next we consider what happens in the less interesting regime
$m<\beta-1$.
 
 \medskip

\begin{remark} If \eqref{tail slow} 
holds with $m=\beta-1$, then Theorems \ref{main theorem} and 
\ref{thm converge 2}
provide the conclusion of Theorem \ref{theorem trivial} even though $ED^m$ may be infinite.
\end{remark}
\begin{proof}[Proof of Theorem \ref{theorem trivial}]
From \eqref{was 1} 
with $n=k-1$ and the
  dominated convergence 
theorem, we see that
\begin{equation}\label{was called 9}
\p(S_D=k-1)\sim p^{k-1}E\binom D{k-1}\quad\text{as }p\downarrow
0.\end{equation} 
We need to show that
$$\lim_{p\rightarrow 0}\frac{\p(S_D\geq k)}{p^{k-1}}=0.$$
This will follow from \eqref{was 7} 
and the dominated convergence theorem provided
that

\begin{equation}\label{was called 11} 1-\sum_{n=0}^{k-1}\binom ln
p^n(1-p)^{l-n}\leq C(lp)^{k-1}\end{equation}
for some $C$ depending only on $k$, and
\begin{equation}\label{was called 12} 1-\sum_{n=0}^{k-1}\binom ln
p^n(1-p)^{l-n}=o(p^{k-1})\end{equation}
as $p\rightarrow 0$ for each $l$. 
Both \eqref{was called 11} and \eqref{was called 12} 
follow from
\begin{equation}\label{was called 13} 1-\sum_{n=0}^{k-1}\binom ln
p^n(1-p)^{l-n}\leq C(lp)^{k}\end{equation}
for some (different) constant $C$, again depending only on $k$; \eqref{was called 13}
is  a Chernoff bound; see \cite[formula (12)]{hagerup}.
That 
\eqref{was called 12} 
follows from \eqref{was called 13} 
is immediate. To deduce \eqref{was called 11} 
from \eqref{was called 13} 
write 
\begin{equation}\label{was called 14}
1-\sum_{n=0}^{k-1}\binom ln
p^n(1-p)^{l-n}=1-\sum_{n=0}^{k-2}\binom ln
p^n(1-p)^{l-n}-\binom l{k-1}p^{k-1}(1-p)^{l-k+1}\end{equation}
and apply \eqref{was called 13} 
to the first part of \eqref{was called 14} 
with $k$ replaced by
$k-1$.
 
The final statement follows from \eqref{was called 9} 
with $k=m+1$.
\end{proof}

\section{Infinite divisibility}

We will show that the distributions in \eqref{G at m} 
are \emph{infinitely divisible};  this is relatively easy, 
thanks to a result from renewal theory.

\begin{prop}\label{inf prop} Suppose the sequence $\{u(n), n\geq 0\}$ satisfies $u(0)=1$,
\begin{equation}\label{hyp}u(n)>0, u(n-1)u(n+1)\geq u^2(n)\text{ for }n\geq 1\text{ and }\lim_n\frac{u(n)}{u(n+1)}>0.\end{equation}
Let 
\begin{equation}\label{series}\log\bigg(\sum_{n=0}^\infty u(n)\x^n\bigg)=\sum_{n=1}^\infty \lambda(n)\x^n.\end{equation}
Then $\lambda(n)\geq 0$ for  $n\geq 1$.
\end{prop}

\begin{proof} Let $\{f(n), n\geq 1\}$ be the sequence associated to $u(\cdot)$ by the renewal equation:
\begin{equation}\label{renewal}u(n)=\sum_{k=1}^nf(k)u(n-k),\end{equation}
and consider the two generating functions
$$U(\x)=\sum_{n=0}^\infty u(n)\x^n\text{ and } F(\x)=\sum_{n=1}^\infty f(n)\x^n.$$
Multiplying (\ref{renewal}) by $\x^n$ and summing for $n\geq 0$ gives
$$U(\x)=1+U(\x)F(\x), \text{ or equivalently } U(\x)=\frac1{1-F(\x)}.$$
Therefore, (\ref{series}) can be written as
$$\log [U(\x)]=-\log(1-F(\x))=\sum_{n=1}^\infty\frac{[F(\x)]^n}n.$$
Kaluza (\cite{kaluza}) proved that $f(k)\geq 0$ for all $k\geq 1$. (See  \cite[Theorem 1]{liggettkarlin} for generalizations of this statement; see also \cite{shanbhag}.)
Therefore the series in (\ref{series}) has nonnegative coefficients.
\end{proof}

The inequality in (\ref{hyp}) is known as log-convexity of the sequence $u$. There is a long history of connections between log-convexity  and infinite divisibility;  see \cite{steutel} \cite{wardekatti} and \cite[Thm. 51.3; Notes on p. 426]{sato}, for example.
\begin{cor} 
For $m=1,2,\ldots$, and $\alpha \in (0,m)$, the probability distribution for $D$ specified by 
\eqref{G at m} and \eqref{our ratio alpha} is infinitely divisible.
\end{cor}
\begin{proof}
Let $X=D-m$, and  define $u(n)  = \p(X=n)/\p(X=0)$ for $n \ge 0$.  This yields
$$ 
  \sum_{n=0}^\infty u(n)\x^n = \bigg[1-(1-\x)^\alpha-\sum_{k=0}^{m-1}(-1)^k\frac{(\alpha)_k}{k!}\x^k\bigg]\bigg/(-1)^m\frac{(\alpha)_m}{m!}\x^m,
$$  
\ignore{
and
$$u(n)=(-1)^n\frac{m!(\alpha-m)(\alpha-m-1)\cdots(\alpha-n-m+1)}{(n+m)!}$$ $$=\frac{(m-\alpha)(m-\alpha+1)\cdots(m-\alpha+n-1)}{(m+1)(m+2)\cdots(m+n)}.$$
Then $u(0)=1$, $u(n)>0$ for all $n>0$ 
$$\frac{u(n)}{u(n+1)}=\frac{m+n+1}{m+n-\alpha},$$
}
so that $u(0)=1$, $u(n)>0$ for all $n>0$ and
\mbox{$u(n)/u(n+1)=(m+n+1)/(m+n-\alpha)$}, 
which is decreasing in $n$, so that (\ref{hyp}) is satisfied.  
The  probability generating function of $X$ is
$$
    G_X(\x) := \e \x^X = \p(X=0) \sum_{n=0}^\infty u(n)\x^n 
= \p(X=0) \exp\left( \sum_{n=1}^\infty \lambda(n)\x^n \right),
$$
and Proposition \ref{inf prop} shows that $\lambda(n) \ge 0$ for $n=1,2,\ldots$. 
Hence $X$ is equal in distribution to $\sum_{n \ge 1} n Z_n$, where $Z_1,Z_2,\ldots$ are independent, and $Z_n$ is Poisson distributed with parameter $\lambda(n)$.
\end{proof}

 \bibliographystyle{plain}
\bibliography{Thinbib}

\ignore{

}    

\end{document}